\documentclass[11pt]{amsart}
\usepackage{mathrsfs,latexsym,amsfonts,amssymb}
\newtheorem{theorem}{Theorem}[section]
\newtheorem{lemma}[theorem]{Lemma}
\newtheorem{corollary}[theorem]{Corollary}
\newtheorem{question}[theorem]{Question}
\newtheorem{remark}[theorem]{Remark}
\theoremstyle{definition}

\newtheorem{proposition}[theorem]{Proposition}

\begin{document}

\title[The quasi-metrizability of hyperspaces]
{The quasi-metrizability of hyperspaces}

\author{Chuan Liu}
\address{(Chuan Liu)Department of Mathematics,
Ohio University Zanesville Campus, Zanesville, OH 43701, USA}
\email{liuc1@ohio.edu}

\author{Fucai Lin}
  \address{(Fucai Lin): School of mathematics and statistics,
  Minnan Normal University, Zhangzhou 363000, P. R. China}
  \email{linfucai@mnnu.edu.cn; linfucai2008@aliyun.com}

\thanks{The second author is supported by the Key Program of the Natural Science Foundation of Fujian Province (No: 2020J02043) and the NSFC (No. 11571158).}

\keywords{quasi-metrizable space; hyperspace; metrizable space; $\gamma$-space; hemicompact}
\subjclass[2010]{primary 54B20; secondary 54E35, 54E45}

\begin{abstract}
For a space $X$, let $(CL(X), \tau_V)$, $(CL(X), \tau_{locfin})$ and $(CL(X), \tau_F)$ be the set $CL(X)$ of all nonempty closed subsets of $X$ which are endowed with Vietoris topology, locally finite topology and Fell topology respectively. We prove that $(CL(X), \tau_V)$ is quasi-metrizable if and only if $X$ is a separable metrizable space and the set of all non-isolated points of $X$ is compact, $(CL(X), \tau_{locfin})$ is quasi-metrizable or symmetrizable if and only if $X$ is metrizable and the set of all non-isolated points of $X$ is compact, and $(CL(X), \tau_F)$ is quasi-metrizable if and only if $X$ is hemicompact and metrizable. As an application, we give a negative answer to a Conjecture in \cite{LL2022}.
\end{abstract}

\maketitle

\section{Introduction and Preliminaries}
A function $d: X\times X\rightarrow \mathbb{R}^{+}$ is called a {\it metric} on the set $X$ if for each $x, y, z\in X$,

\smallskip
(i)  $d(x, y)=0$ iff $x=y$;

\smallskip
(ii) $d(x, y)=d(y, x)$;

\smallskip
(iii) $d(x, z)\leq d(x, y)+d(y, z)$.\\

The function $d$ is called a symmetric (or quasi-metric) if $d$ satisfies $(i'), (ii)$  $((i), (iii))$ resctively, where $(i'): d(x, y)=0$ implies $x=y$. A topological space $(X, \mathscr{F})$ is said to be {\it quasi-metrizable} if there exists a quasi-metric on $X$ compatible with $\mathscr{F}$ (i.e., the family of all $\varepsilon$-neighborhoods forms a base for the topology). A space $X$ is said to be symmetrizable if there is symmetric $d$ on $X$ satisfying the following condition: $U\subset X$ is open if and only if for each $x\in U$, there exists $\epsilon >0$ with $B(x, \epsilon)\subset U$, where $B(x,\epsilon)=\{y\in X: d(x, y)<\epsilon\}$.

Quasi-metric spaces are roughly speaking like metric spaces without symmetry. There are many applications of quasi-metric to the theory of topology, mathematics analysis, probability and theoretical computer science which can be found e.g.
in \cite{B1997, Br2003, C2012, H1999, RT2011}. In \cite{LL2022}, we give the following Conjecture 1:

\smallskip
{\bf Conjecture 1 \cite{LL2023}:} Let $X$ be a space. Then  $(CL(X), \tau_V)$ is quasi-metrizable if and only if $X=C\oplus D$, where $X$ is a compact metric space and $D$ is a countable discrete space. 

\smallskip
In this paper, we give a negative answer to Conjecture 1 above, see Theorem ~\ref{t1}. Moreover, in \cite{LL2022} the authors proved that $(CL(X), \tau_V)$ is symmetrizable if and only if $X$ is compact metrizable. However, this is not true for quasi-metrizable.
In this paper, we mainly give some characterizations of the quasi-metrizability of hyperspaces. Given a topological space $X$, we define its {\it hyperspace} as the
following sets:

\smallskip
$CL(X)=\{A\subset X: A \ \mbox{is nonempty closed in}\ X\}$,

\smallskip
$\mathcal{K}(X)=\{K\subset X: K\ \mbox{is nonempty compact}\},$

\smallskip
$\mathcal{F}_n(X)=\{A\subset X: 1\leq|A|\leq n\}, n\in\mathbb{N}$ and

\smallskip
$\mathcal{F}(X)=\{A\subset X: A\neq\emptyset\ \mbox{and}\ A\ \mbox{is finite}\}.$

\smallskip
Obviously, we have $\mathcal{F}_n(X)\subset \mathcal{F}(X)\subset \mathcal{K}(X)\subset CL(X)$ and
$\mathcal{F}(X)={\bigcup_{n=1}^{\infty}}\mathcal{F}_n(X)$.

If $\mathcal{U}$ is a nonemptyset family of subsets of a space $X$, let
$$\mathcal{U}^-=\{A\in CL(X): A\cap U\neq\emptyset, U\in
\mathcal{U}\}$$ and $$\mathcal{U}^+=\{A\in CL(X): A\subset
\bigcup\mathcal{U}\};$$in particular, if $\mathcal{U}=\{U\}$, then we denote $\mathcal{U}^-$ and $\mathcal{U}^+$ by $U^-$ and $U^+$ respectively.

For any open subsets $U_1, \ldots, U_k$ of space $X$, let $$\langle
U_1, ..., U_k\rangle =\{H\in CL(X): H\subset {\bigcup}_{i=1}^k U_i\
\mbox{and}\ H\cap U_j\neq \emptyset, 1\leq j\leq k\}.$$ We endow
$CL(X)$ with {\it Vietoris topology} defined as the topology
generated by $$\{\langle U_1, \ldots, U_k\rangle: U_1, \ldots, U_k\
\mbox{are open subsets of}\ X, k\in \mathbb{N}\}.$$  We endow
$CL(X)$ with {\it Fell topology} defined as the topology generated
by the following two families as a subbase:
$$\{U^{-}: U\ \mbox{is any non-empty open in}\ X\}$$ and $$\{(K^{c})^{+}: K\ \mbox{is a compact subset of}\ X, K\neq X\}.$$
The {\it locally finite topology} $\tau_{loc fin}$
on $CL(X)$ has a subbase consisting of all subsets of the forms $V^+$
and $\mathcal{U}^-$, where $V$ ranges over all open subsets of $X$
and $\mathcal{U}$ range over all the locally finite families of open
subsets of $X$. We denote the hyperspace $CL(X)$ with the Vietoris topology, the Fell
topology and the locally finite topology by $(CL(X), \tau_{V})$, $(CL(X), \tau_{F})$ and $(CL(X), \tau_{locfin})$
respectively. For any $T_2$-space $X$, it is well known that $F_1(X)$ is
homeomorphic to $X$ in hyperspaces with the Vietoris topology, the
Fell topology and the locally finite topology respectively, so we consider all hyperspaces have a closed copy of
$X$.

Let $I(X)=\{x\in X: x \ \mbox{is an isolate point of $X$}\}$,
$NI(X)=X\setminus I(X)$; then $X=I(X)\cup NI(X)$. Throughout this paper, all topological spaces are assumed to be
$T_{2}$, unless otherwise is explicitly stated. Let $\mathbb{N}$ be the set of all positive integers and $\omega$ the first infinite ordinal. For undefined notations and
terminologies, the reader may refer to \cite{E1989},  \cite{G1984}
and \cite{M1951}.

\section{main results}
In this section, we shall give some characterizations of a space $X$ such that $(CL(X), \tau_{V})$, $(CL(X), \tau_{F})$ and $(CL(X), \tau_{locfin})$ are quasi-metric respectively. First, we give some lemmas and definitions.

\begin{lemma}\label{l1}\cite[Lemma 2.3]{LL2023}
Let $K$ be a compact subset of a topological space $X$ and $\{U_i:
i\leq k\}$ be an open cover of $K$. Then, for each $i\leq k$, there
exists a compact subset $K_i$ of $X$ such that $K_i\subset U_i$,
$K=\bigcup_{i\leq k}K_i$ and $K_i\neq \emptyset$ whenever $U_i\cap
K\neq \emptyset$.
\end{lemma}

\begin{lemma}\label{l2}
(\cite[Theorem 10.2]{G1984}) A space $X$ is quasi-metrizable if and
only if there is a function $g: \omega\times X\to \tau$ such that
(i) $\{g(n, x): n\in \omega\}$ is a base at $x$ for each $x\in X$; (ii) $y\in g(n+1,
x)\Rightarrow g(n+1, y)\subset g(n, x)$ for each $x\in X$ and $n\in\mathbb{N}$.

\end{lemma}

\begin{lemma}\label{l3}
\cite[Lemma 2.3.1]{M1951} $\langle U_1, ..., U_n\rangle\subset
\langle V_1,..., V_m\rangle$ if and only if
$\bigcup_{j=1}^nU_j\subset \bigcup_{j=1}^mV_j$ and for every $V_i$
there exists $U_k$ such that $U_k\subset V_i$.
\end{lemma}

\begin{lemma}\label{l4}
Let $X$ be a countably compact space, if each closed countably
compact subset of $(CL(X), \tau_V)$ is a $G_\delta$-set, then $X$ is
compact metrizable.
\end{lemma}

\begin{proof}
Since $\mathcal{F}_1(X)$ is homeomorphic to $X$, it follows that
$\mathcal{F}_1(X)$ is a closed countably compact subset of $(CL(X),
\tau_V)$, hence it is a $G_\delta$-set in $(CL(X), \tau_V)$. By the proof
of \cite[Lemma 2.5]{LL2023}, we conclude that $X$ has a $G_\delta$-diagonal, hence
$X$ is compact metrizable by \cite[Theorem 2.14]{G1984}.
\end{proof}

\begin{remark}
 Let $X$ be countably compact; from \cite[Theorem
1.2]{NS1988}, it follows that $(CL(X), \tau_V)=(CL(X), \tau_{loc fin})$. Therefore, if
each closed countably compact $(CL(X), \tau_{loc fin})$ is a
$G_\delta$-set, then $X$ is compact metrizable.
\end{remark}

A space $(X, \tau)$ is a {\it $\gamma$-space} if there exists a
function $g: \omega\times X \to \tau$ such that (i) $\{g(n, x): n\in
\omega\}$ is a base at $x$ for each $x\in X$; (ii) for each $n\in \omega$ and $x\in
X$, there exists $m\in \omega$ such that $y\in g(m, x)$ implies
$g(m, y)\subset g(n, x)$.

Recalled that a space $X$ is a $D_1$-space \cite{A1966} if every
closed subset of $X$ has a countable local base.

Now we can prove one of main theorems of our paper.

\begin{theorem}\label{t1}
The following statements are equivalent for a space $X$.

\begin{enumerate}
\item $(CL(X), \tau_V)$ is quasi-metrizable;

\smallskip
\item $(CL(X), \tau_V)$ is a $\gamma$-space;

\smallskip
\item Each closed countably compact subset of $(CL(X), \tau_V)$ has
a countable base;

\smallskip
\item $X$ is separable, metrizable and $NI(X)$ is compact.
\end{enumerate}

\end{theorem}

\begin{proof}
The implication (1)$\Rightarrow$ (2) is trivial. The implication (2) $\Rightarrow$ (3) by \cite[Corollary 10.8 (ii), Theorem 10.6(iii)]{G1984}. We only need to prove (3) $\Rightarrow$ (4) and (4) $\Rightarrow$ (1).

(3) $\Rightarrow$ (4). Since $(CL(X), \tau_V)$ is first-countable, it follows that $X$
is separable and a $D_1$-space \cite[Theorem 5.15]{LL2022}. Hence $NI(X)$
is countably compact by \cite[Theorem 1]{DL1995}, then $NI(X)$ is
compact and metrizable by Lemma ~\ref{l4}. Hence $X$ is separable and metrizable
\cite[Theorem 7]{DL1995}.

(4) $\Rightarrow$ (1). Let $(X, d)$ be a separable metrizable space with $NI(X)$ compact.
If $X$ is compact metrizable, then it follows from \cite[Theorem
3.3]{H2003} that $(CL(X), \tau_V)$ is compact metrizable, hence it
is quasi-metrizable. Now we assume $X$ is not compact. Then, by
induction, we can construct a sequence of finite open covers
$\{\mathcal{G}_n: n\in\mathbb{N}\}$ of $NI(X)$ and a sequence of
finite compact covers $\{\mathcal{K}_{n}: n\in\mathbb{N}\}$ of
$NI(X)$ such that the following (a)-(c) hold.

\smallskip
(a) For each $n\in\mathbb{N}$, the family $\mathcal{G}_n=\{B(x_{n,
i}, 1/2^{q_n}): i\leq l_n\}$ is a cover of $NI(X)$, where $l_n$ is a
finite subset of $\mathbb{N}$,  $\{x_{n, i}: i\leq l_n\}\subset
NI(X)$ and $q_{n}<q_{n+1}$;

\smallskip
(b) For each $n\in\mathbb{N}$, each element of $\mathcal{G}_{n+1}$
is contained in some element in $\mathcal{G}_n$, and each element of
$\mathcal{G}_n$ contains some element of $\mathcal{G}_{n+1}$.

\smallskip
(c) For each $n\in\mathbb{N}$, $\mathcal{K}_{n}=\{K_i(n): i\leq
l_n\}$ such that $K_i(n)\subset B(x_{n, i}, 1/2^{q_n})$  for each
$i\leq l_n$ and $NI(X)=\bigcup \{K_i(n): i\leq l_n\}$, where each
$K_i(n)$ is compact.

\smallskip
Indeed, the family $\{B(x, 1/2): x\in NI(X)\}$ is an open cover of
$NI(X)$; then since $NI(X)$ is compact, without loss of generality,
we may assume that each $B(x, 1/2)\neq X$, then there exists a
minimal and finite subset $\{x_{1, i}: i\leq l_1\}$ of $NI(X)$ such
that $\mathcal{G}_1=\{B(x_{1, i}, 1/2): i\leq l_1\}$ is a finite
cover of $NI(X)$. Moreover, let $q_1=1$. From Lemma~\ref{l1}, there
is finite cover $\mathcal{K}_{1}=\{K_i(1): i\leq l_1\}$ of $NI(X)$
consisting of compact subsets of $NI(X)$ such that $K_i(1)\subset
B(x_{1, i}, 1/2^{q_1})$ and $NI(X)=\bigcup \{K_i(1): i\leq l_1\}$.
Assume that $\mathcal{G}_n=\{B(x_{n, i}, 1/2^{q_n}): i\leq l_n\}$
and $\mathcal{K}_{n}=\{K_i(n): i\leq l_n\}$ have been constructed,
where $l_{n}, q_{n}\in\mathbb{N}$ and $x_{n, i}\in NI(X)$ for each
$i\leq l_{n}$ and $NI(X)=\bigcup_{i\leq l_n}K_i(n)$. Let $q_{n+1}$ be
the smallest natural number such that $q_{n+1}>q_n$ and
 $$1/2^{q_{n+1}}\leq d(K_i(n), X\setminus B(x_{n, i}, 1/2^{q_n}))/4$$ for all $i\leq l_n$. Fix any $i\leq l_{n}$. Then
$\{B(x, 1/2^{q_{n+1}}): x\in K_i(n)\}$ is an open cover of $K_i(n)$
and $B(x, 1/2^{q_{n+1}})\subset B(x_{n, i}, 1/2^{q_n})$ for each
$x\in K_i(n)$, hence there is a minimal and finite subcover
$\mathcal{B}'_{n+1, i}$ of $\{B(x, 1/2^{q_{n+1}}): x\in K_i(n)\}$ in
$K_i(n)$. Now let $\mathcal{G}_{n+1}=\bigcup\{\mathcal{B}'_{n+1, i}:
i\leq l_n\}$. From Lemma~\ref{l1}, there are finite compact subsets
$\{K_i(n+1): i\leq l_{n+1}\}$ such that $K_i(n+1)\subset B(x_{n+1,
i}, 1/2^{q_n+1})$ and $NI(X)=\bigcup \{K_i(n+1): i\leq l_{n+1}\}$.
Then it is straightforward to verify that $\{\mathcal{G}_n: n\in
\mathbb{N}\}$ and $\{\mathcal{K}_{n}: n\in\mathbb{N}\}$ satisfy
 (a)-(c).

Now, for a closed subset $K\subset NI(X)$, let
$$\mathcal{G}'_n(K)=\{B(x_{n, i}, 1/2^{q_n})\in \mathcal{G}_n:
B(x_{n, i}, 1/2^{q_n})\cap K\neq \emptyset\};$$ then
$\mathcal{G}'_n(K)$ is an open cover of $K$. Clearly,
$\{\mathcal{G}'_{n}(K): n\in \mathbb{N}\}$ satisfies (1)-(2) above.
For each $n\in\mathbb{N}$ and closed subset $K\subset NI(X)$, we
write $\mathcal{G}'_n(K)=\{U_{n, m}(K): m\leq r_{n, K}\}$, where
$r_{n, K}\in\mathbb{N}$. Moreover, for any closed subsets $K$ and
$L$ in $NI(X)$ with $K\subset L$, we have $\mathcal{G}'_n(K)\subset
\mathcal{G}'_n(L)$. First, we prove that the following claim holds.

\smallskip
{\bf Claim:} For any compact subsets $K$ and $L$ in $NI(X)$, if
$K\subset\bigcup\mathcal{G}'_{n+1}(L)$ for some $n\in\mathbb{N}$,
then, for any $m\leq r_{n+1, K}$, we have $U_{n+1, m}(K)\subset
U_{n, j}(L)$ for some $U_{n, j}(L)\in\mathcal{G}'_{n}(L)$.

\smallskip
Since $U_{n+1, m}(K)\cap K\neq \emptyset$ and
$K\subset\bigcup\mathcal{G}'_{n+1}(L)$, we have $U_{n+1, m}(K)\cap
U_{n+1, m'}(L)\neq \emptyset$ for some $m'\leq r_{n+1, L}$; then
$U_{n+1, m'}(L)\in \mathcal{B}'_{n+1, i}$, and it is obvious that
$U_{n+1, m'}(L)\cap K_i(n)\neq\emptyset$ for some $i\in\mathbb{N}$ by our
construction. Pick any $U_{n, j}(L)\in\mathcal{G}'_{n}(L)$ such that
$K_i(n)\subset U_{n, j}(L)$ in the construction. Fix any $y\in
U_{n+1, m}(K)$; then $d(K_i(n), y)\leq d(x, y)$ for any $x\in K_i(n)\cap
U_{n+1, m'}(L)$. Pick any $x'\in U_{n+1, m}(K)\cap U_{n+1, m'}(L)$;
then $d(K_i(n), y)\leq d(x, y)\leq d(x, x')+d(x',
y)<(1/2^{q_{n+1}}+1/2^{q_{n+1}})\times 2\leq 4\times d(K_i(n),
X\setminus U_{n, j}(L))/4 =d(K_i(n), X\setminus U_{n, j}(L))$, thus
$y\in U_{n, j}(L)$. Therefore, we have $U_{n+1, m}(K)\subset U_{n,
j}(L)$. The proof of Claim is completed.

Since $X=I(X)\cup NI(X)$ is separable and metrizable, it follows
that $I(X)$ is a countable set. Hence let $I(X)=D=\{d_n:
n\in\mathbb{N}\}$. Moreover, for any closed subset $A$ of $X$, there
exist subset $D_1(A)$ of $D$ and a closed subset $K(A)$ of $NI(X)$
such that $A=D_1(A)\cup K(A)$. Next we define a $g$-function $G(n,
A)$ satisfies Lemma~\ref{l2}. We consider the following two cases.

{\bf Case 1:} $X=D\oplus NI(X)$, where $D$ is a countable discrete
subspace and $NI(X)$ is a compact metrizable space. Without loss of
generality, we may assume $d(D, NI(X))=1$. Take any closed subset
$A$ of $X$. We define $G(k, A)$ as follows (1$^{'}$)-(3$^{'}$):

\smallskip
(1$^{'}$) If $D_1(A)=\emptyset$, then put $$G(k, A)=\langle U_{k,
1}(K(A)), ..., U_{k, r_{k, K(A)}}(K(A))\rangle.$$

\smallskip
(2$^{'}$) If $D_{1}(A)$ is a finite set, then let
$D_1(A)=\{d_{n_i(A)}: i\leq r(A)\}\neq\emptyset$, where each
$n_i(A)\in\mathbb{N}$ and $r(A)\in\mathbb{N}$; now put $$G(k,
A)=\langle \{d_{n_1(A)}\}, ..., \{d_{n_{r(A)}(A)}\}, U_{k, 1}(K(A)),
..., U_{k, r_{k, K(A)}}(K(A))\rangle$$ when $K\neq\emptyset$, and put
$$G(k, A)=\langle \{d_{n_1(A)}\}, ..., \{d_{n_{r(A)}(A)}\}\rangle$$
when $K=\emptyset$.

\smallskip
(3$^{'}$) If $D_{1}(A)$ is an infinite set, then let
$D_1(A)=\{d_{n_{j}(A)}: j\in\mathbb{N}\}$ such that $\{n_{j}(A):
j\in\mathbb{N}\}$ is an infinite subsequence of $\mathbb{N}$ with
$n_{i}(A)<n_{j}(A)$ if $i<j$; now put $$G(k, A)=\langle
\{d_{n_{1}(A)}\}, ..., \{d_{n_{k}(A)}\}, D_1(A), U_{k, 1}(K(A)),
..., U_{k, r_{k, K(A)}}(K(A))\rangle$$ when $K\neq\emptyset$, and put
$$G(k, A)=\langle \{d_{n_1(A)}\}, ..., \{d_{n_{r(A)}(A)}\}\rangle$$
when $K=\emptyset$.

\smallskip
We prove $G(n, A)$ satisfies (i) and (ii) in Lemma ~\ref{l2}. We
divide the proof into the following cases.

\smallskip
{\bf Subcase 1.1:} $D_1(A)$ is infinite and $K(A)\neq \emptyset$.

\smallskip
First, we prove $G(n, A)$ satisfies (i). For each $n\in\mathbb{N}$,
put $D_{1n}(A)=\{d_{n_i(A)}: i\leq n\}$. Let $\langle V_1, ...,
V_l\rangle$ be a neighborhood of $A$ in $(CL(X), \tau_V)$; then each
$V_i$ either contains some $d_{n_i(A)}$ or $V_i\cap K(A)\neq
\emptyset$. Since $K(A)\subset \bigcup\{V_i: V_i\cap K(A)\neq
\emptyset\}$, we rewrite $\{V_i: V_i\cap K(A)\neq
\emptyset\}=\{V_{i_j}: j\leq p\}$, where $p\in\mathbb{N}$ and $p\leq
l$. By Lemma ~\ref{l1}, for each $j\leq p$, there is a nonempty
compact subset $K_j(A)\subset V_{i_j}$ such that $\bigcup\{K_{j}(A):
j\leq p\}=K(A)$. For each $j\leq l$ with $V_{j}\cap K(A)=\emptyset$,
put $m_{j}=\min\{n_{i}: d_{n_{i}}\in V_{j}\}$. Pick any
$n\in\mathbb{N}$ such that $1/2^n <\min\{d(K_j(A), X\setminus
V_{i_j})/2: j\leq p\}$ and $n>\max\{m_{j}: j\leq l, V_{j}\cap
K(A)=\emptyset\}$; then $G(n, A)\subset \langle V_1, ...,
V_l\rangle$. Indeed, for each $j\leq p$, if $B(x, 1/2^n)\cap
K_j(A)\neq\emptyset$, then $B(x, 1/2^n)\subset V_{i_j}$; since
$\mathcal{G}'_n(K(A))=\bigcup_{j\leq p}\mathcal{G}'_n(K_{j}(A))$, it
follows that each element $U_{n, m}(K(A))$ in $\mathcal{G}'_n(K(A))$
is contained in some $V_{i_j}$, and each $V_{i_j}$ contains some
element $U_{n, m}(K(A))$ of $\mathcal{G}'_n(K(A))$. Then it is easy
to see that $$D_{1}(A)\cup(\bigcup\{U_{n, m}(K(A)): m\leq r_{n,
K(A)}\}) \subset \bigcup\{V_i: i\leq l\}$$ and each $V_i$ contains
either some $d_{n_{i}(A)}$ or $U_{n, m}(K(A))$. Therefore, $G(n,
A)\subset\langle V_1, ..., V_l\rangle$ by Lemma ~\ref{l3}.
Therefore, $\{G(n, A): n\in\mathbb{N}\}$ is a local base of $A$ in
$(CL(X), \tau_V)$ .

Now, we prove $G(n, A)$ satisfies (ii). Take any $B=D_1(B)\cup
K(B)\in G(n+1, A)$. We prove that $G(n+1, B)\subset G(n, A)$.
Clearly, $D_{1}(B)\subset D_{1}(A)$, $K(B)\subset
\bigcup\mathcal{G}'_{n+1}(K(A))$ and $U_{n+1, m}(K(A))\cap K(B)\neq
\emptyset$ for each $m\leq r_{n+1, K(A)}$. Hence
$\mathcal{G}'_{n+1}(K(A))\subset \mathcal{G}'_{n+1}(K(B))$. From
Claim, it follows that, for any $m\leq r_{n+1, K(B)}$, we have
$U_{n+1, m}(K(B))\subset U_{n, j}(K(A))$ for some $U_{n,
j}(K(A))\in\mathcal{G}'_{n}(K(A))$. Then
$$\bigcup\mathcal{G}'_{n+1}(K(B))\subset \bigcup
\mathcal{G}'_n(K(A)).$$ Moreover, since each element of
$\mathcal{G}'_{n}(K(A))$ contains some element of
$\mathcal{G}'_{n+1}(K(A))$ and $\mathcal{G}'_{n+1}(K(A))\subset
\mathcal{G}'_{n+1}(K(B))$, we conclude that each element of
$\mathcal{G}'_{n}(K(A))$ contains some element of
$\mathcal{G}'_{n+1}(K(B))$. Now, by Lemma ~\ref{l3} and
$D_{1(n+1)}\subset B$, we have $G(n+1, B)\subset G(n, A)$.

The proofs of the following two cases are similar to Subcase 1.1,
thus we omit the proofs of them.

\smallskip
{\bf Subcase 1.2:} $D_1$ is finite and $D_1(A)\neq \emptyset$,
$K(A)\neq \emptyset$.

\smallskip

\smallskip
{\bf Subcase 1.3:} $D_{1}(A)=\emptyset$ or $K(A)=\emptyset$.

\smallskip
{\bf Case 2:} $X=D\cup NI(X)$ such that $D$ is a countable discrete
subspace and any open neighborhood of $NI(X)$ is not compact.

\smallskip
Then we can choose a countably and decreasingly local base $\{U_n:
n\in\mathbb{N}\}$ of $NI(X)$ in $X$ such that $U_1=X$,
$|(U_n\setminus U_{n+1})\cap D|=\omega$ for each $n\in\mathbb{N}$.
For each $n\in\mathbb{N}$, let $D_n=(U_n\setminus U_{n+1})\cap D$
and enumerate it as $\{a_{n, j}: j\in \mathbb{N}\}$. We rewrite
$D=\{a_m: m\in \mathbb{N}\}$ such that, for each $m\in\mathbb{N}$,
we have $a_m=a_{i, j}$, where each $m=(i+j-1)(i+j-2)/2+j$, $i,
j\in\mathbb{N}$. Moreover, for each $n$, put $d(X\setminus U_n,
NI(X))=s_n>0$; then $s_i>s_j$ for any $i<j$.

Let $A$ be an arbitrary closed subset of $X$, $A=D_{1}(A)\cup K(A)$
with $K(A)\subset NI(X)$ compact and $D_{1}(A)=\{a_{n_k}: k\leq
r(A)\}\subset D$ for some $r(A)\in\mathbb{N}$ or
$D_{1}(A)=\{a_{n_k}: k\in \mathbb{N}\}\subset D$, where
$n_k<n_{k+1}$ for each $k$.

For any $m$, let $$p_m=\max\{i: a_{n_k}=a_{i, j}, k\leq m\}$$ and
$$q(m)=\min\{j: 1/2^{q_j}<s_{p_m}\}.$$ Then
$\mathcal{G}'_{q(m)}(K(A))=\{U_{q(m), 1}(K(A)), ..., U_{q(m),
r_{q(m), K(A)}}(K(A))\}$. We define $G(k, A)$ as follows
(1$^{''}$)-(2$^{''}$):

\smallskip
(1$^{''}$) If $D_1(A)$ is finite, then put $$G(k, A)=\langle
U_{q(k), 1}(K(A)), ..., U_{q(k), r_{q(k), K(A)}}(K(A))\rangle$$ for
each $k\in\mathbb{N}$ if $D_1(A)=\emptyset$; put $$G(k, A)=\langle
\{a_{n_1}\}, ..., \{a_{n_{r(A)}}\}, U_{q(k), 1}(K(A)), ..., U_{q(k),
r_{q(k), K(A)}}(K(A))\rangle$$ if $D_1(A)\neq\emptyset$ and
$K(A)\neq\emptyset$;  put $$G(k, A)=\langle \{a_{n_1}\}, ...,
\{a_{n_{r(A)}}\}\rangle$$ if $K(A)=\emptyset$.

\smallskip
(2$^{''}$) If $D_{1}(A)$ is an infinite set, then put $$G(k,
A)=\langle \{a_{n_{1}}\}, ..., \{a_{n_{k}}\}, D_1(A), U_{q({k}),
1}(K(A)), ..., U_{q({k}), r_{q({k}), K(A)}}(K(A))\rangle$$ if
$K(A)\neq\emptyset$, and put $$G(k, A)=\langle \{a_{n_1(A)}\}, ...,
\{a_{n_{r(A)}(A)}\}\rangle$$ if $K(A)=\emptyset$.

We prove $G(m, A)$ satisfies (i) and (ii) for the case of $D_{1}(A)$
being infinite and $K(A)\neq\emptyset$. The proofs of the other
cases are similar to Case~1.

Let $\langle V_1, ..., V_l\rangle$ be an any open neighborhood of $A$.
Put$$k'=\max\{\min\{j: a_{n_j}\in V_i, V_i\cap D_{1}(A)\neq\emptyset\}: i\leq l\}$$ (note:
$\min\{j: a_{n_j}\in V_i, V_i\cap D'\neq\emptyset\}=0$ if $V_i\cap
D(A)=\emptyset)$. Clearly, $K(A)\subset \bigcup\{V_i: V_i\cap K(A)\neq
\emptyset\}$, and we can enumerate $\{V_i: V_i\cap K(A)\neq
\emptyset\}$ as $\{V_{i_j}: j\leq l^{'}\}$. By Lemma~\ref{l1}, there
exists a collection of compact subsets $\{K_j: j\leq l^{'}\}$ such
that $K(A)=\bigcup \{K_j: j\leq l^{'}\}$ and $K_j\subset V_{i_j}$ for
each $j\leq l^{'}$. Let $k''$ be the smallest natural number such
that $1/2^{k''}<\min\{d(K_j, X\setminus V_{i_j})/2, j\leq l^{'}\}$.
Let $k=\max\{k', k''\}$; then each $V_i$ either contains some
element of $\mathcal{G}'_{q({k})}(K(A))=\{U_{q({k}), j}(K(A)): j\leq
r_{q(k), K(A)}\}$ or $a_{n_i}$ for some $i\leq k$, and each element
of $\mathcal{G}'_{q(p_{k})}(K(A))$ is contained in some $V_i$. By
Lemma ~\ref{l3}, it is easy to see that
$$G(k, A)=\langle \{a_{n_1}\}, ..., \{a_{n_k}\}, D_{1}(A), U_{q({k}), 1}, ..., U_{q({k}), r_{q({k}), K(A)}}\rangle\subset \langle V_1, ..., V_l\rangle.$$ Hence $\{G(k, A): k\in \mathbb{N}\}$ is a local base of $A$.

Now, we prove $G(k, A)$ satisfies (ii). Take any $B\in G(k+1, A)$;
then $B=D_1(B)\cup K(B)$. Clearly, $\{a_{n_i}: i\leq k+1\}\subset
D_1(B)$ and $K(B)$ is a compact subset of $NI(X)$. For any
$U\in\mathcal{G}'_{q(k+1)}(K(A))$, $U\cap K(B)\neq\emptyset$, then
$\mathcal{G}'_{q(k+1)}(K(A))\subset \mathcal{G}'_{q(k+1)}(K(B))$.
Similar to the proof of Subcase 1.2, we can prove that $G(k+1,
B)\subset G(k, A)$.
\end{proof}

From Theorem~\ref{t1}, it follows that Conjecture 1 is negative. Moreover, in \cite{LL2022}However, the following two questions are unknown for us.

\begin{question}
If $X$ is quasi-metrizable, is $(\mathcal{K}(X), \tau_V)$
quasi-metrizable?
\end{question}

Put $F_C(X)=\{S: S$ is a convergent sequence (with limit point) of $X,
|S|=\omega\}$.

\begin{question}\label{q1}
If $X$ is quasi-metrizable, is $(F_C(X), \tau_V)$ quasi-metrizable?
\end{question}

The following proposition gives a partial answer to Question~\ref{q1}.

\begin{proposition}
If $X$ is quasi-metrizable, then $(\mathcal{F}(X), \tau_V)$ is
quasi-metrizablie.
\end{proposition}

\begin{proof}
Let $(X, d)$ be a quasi-metrizable space with a quasi-metric $d$,
and let $B(x, n)=\{y: d(x, y)<1/2^n\}$ for any $x\in X$ and each
$n\in\mathbb{N}$. For any $A=\{x_1, ..., x_k\}\in \mathcal{F}(X)$,
define $G(n, A)$ as follows:
$$G(n, A)=\langle B(x_1, n), ..., B(x_k, n)\rangle.$$

We prove $\{G(n, A): k\in\mathbb{N}\}$ satisfies (i) and (ii) in
Lemma ~\ref{l2}.

First, we prove $\{G(n, A): k\in\mathbb{N}\}$ satisfies (i). Let
$\langle U_1, ..., U_m\rangle$ be an open neighborhood of $A$. Since
$A\subset\bigcup_{j\leq m}U_j$ and $A\cap U_j\neq \emptyset$ $
(j\leq m)$, we can choose a finite family of open subsets $\{V_i:
i\leq k\}$ satisfies the following conditions (1)-(4):

\smallskip
(1) for each $i\leq k$, $x_i\in V_i$;

\smallskip
(2) for any distinct $i, j\leq k$, $V_i\cap V_j=\emptyset$;

\smallskip
(3) $\bigcup_{i\leq k}V_i\subset \bigcup_{j\leq m}U_j$;

\smallskip
(4) each $U_j$ contains some $V_i$.

\smallskip
Choose $n'\in\mathbb{N}$ such that $B(x_i, n')\subset V_i$ for each
$i\leq k$. By Lemma~\ref{l3}, we have $$G(n', A)=\langle B(x_1, n'),
..., B(x_k, n')\rangle\subset \langle V_1, ..., V_k\rangle\subset
\langle U_1, ..., U_m\rangle.$$

Finally, we prove $\{G(n, A): k\in\mathbb{N}\}$ satisfies (ii). If
$B\in G(n+1, A)$, then let $B=\{y_1, ..., y_l\}$. For any $j\leq l$,
we have $y_j\in B(x_{i_j}, n+1)$ for some $i_{j}\leq k$, hence
$B(y_j, n+1)\subset B(x_{i_j}, n)$. Since $B\in G(n+1, A)$, each
$B(x_i, n)$ contains some $B(y_j, n+1)$. Therefore, it follows from
Lemma~\ref{l3} that $G(n+1, B)=\langle B(y_1, n+1), ..., B(y_l,
n+1)\rangle \subset \langle B(x_1, n), ..., B(x_k, n)\rangle =G(n,
A)$.
\end{proof}

Next we prove the second main theorems of our paper. First, we give a proposition.

\begin{proposition}\label{p2}
For a space $X$, $(\mathcal{K}(X), \tau_V)$ is homeomorphic to
$(\mathcal{K}(X), \tau_{loc fin})$.
\end{proposition}

\begin{proof}
It is known that $\tau_V\subset \tau_{loc fin}$. Fix any $K\in
\mathcal{K}(X)$. If $K\in\mathcal{U}^-\cap \mathcal{U}^+$, where
$\mathcal{U}$ is a locally finite family of open subsets, then $K$
meets every element of $\mathcal{U}$, thus $\mathcal{U}$ must be
finite, which means $\mathcal{U}^-\cap \mathcal{U}^+\in \tau_V$,
hence $\tau_{loc fin}\subset \tau_V$.
\end{proof}

\begin{remark}
If $X$ is a $T_1$-space, then it follows from
Proposition ~\ref{p2} that $X$ is homeomorphic to
$(\mathcal{F}_1(X), \tau_{loc fin})$.
\end{remark}

\begin{theorem}\label{t3}
The following statements are equivalent for a space $X$.
\begin{enumerate}
\item $(CL(X), \tau_{loc fin})$ is metrizable;

\smallskip
\item $(CL(X), \tau_{loc fin})$ is symmetrizable;

\smallskip
\item $X$ is metrizable and $NI(X)$ is compact.
\end{enumerate}
\end{theorem}

\begin{proof}
The implications (1) $\Rightarrow$ (2) is trivial. By \cite[Theorem
2.3]{BHPV1987}, we have (3) $\Rightarrow$ (1). Now it suffices to prove (2)
$\Rightarrow$  (3).

(2) $\Rightarrow$ (3). By Proposition  ~\ref{p2}, $X$ is
symmetrizable. Fix $x\in X$. Clearly, we have $\{x\}\in
Cl(\{\overline{V}: V$ is an open neighborhood of $x\ \mbox{in}\
X\})$; Since $(CL(X), \tau_{loc fin})$ has countable tightness,
there is a countable subfamily $\{V_n: n\in\mathbb{N}\}\subset \{V:
V$ is an open neighborhood of $x\ \mbox{in}\ X\}$ such that
$\{x\}\in Cl(\{\overline{V_n}: n\in\mathbb{N}\})$. We claim that
$\{V_n: n\in\mathbb{N}\}$ is a countable local base of $x$. Indeed,
take any open neighborhood $U$ of $x$ in $X$, since $U^+$ is an open
neighborhood of $\{x\}$ in $(CL(X), \tau_{loc fin})$ and $\{x\}\in
Cl(\{\overline{V_n}: n\in\mathbb{N}\})$, there exists
$n\in\mathbb{N}$ such that $V_n\subset \overline{V_n}\in U^+$, which
implies that $V_n\subset U$. Hence $X$ is a semi-metrizable space
\cite[Theorem 9.6]{G1984}. Next it suffices to prove that $X$ is
$D_1$-space. Indeed, since $X$ is a $D_{1}$-space, then it follows
from \cite[Theorem 1]{DL1995} that $NI(X)$ is countably compact.
From the semi-metrizability of $X$, it follows from \cite[Theorem
9.8]{G1984} that $X$ is semi-stratifiable, then $NI(X)$ is
metrizable \cite[Theorem 5.11, Theorem 2.14]{G1984}. Hence $X$ is
metrizable by \cite[Theorem 7]{DL1995}.

Fix any $A\in CL(X)$. Since $(CL(X), \tau_{loc fin})$ is
symmetrizable, we may assume that $\{\mathcal{W}_n: n\in
\mathbb{N}\}$ is a countable weak base at $A$ in $(CL(X), \tau_{loc
fin})$, then $\bigcup\mathcal{W}_n$ is a sequential neighborhood of
$A$ in $X$ for each $n\in\mathbb{N}$. Indeed, for any sequence
$\{x_k: k\in \mathbb{N}\}$ with $x_k\to x\in A$ in $X$, we have
$A\cup\{x_k\}\to A$ in $(CL(X), \tau_{loc fin})$, then there is
$m\in \mathbb{N}$ such that $\{A\cup\{x_k\}: k\geq m\}\subset
\mathcal{W}_n$, hence $\{x_k: k\geq m\}\subset
\bigcup\mathcal{W}_n$. Therefore,
$\{\mbox{int}(\bigcup\mathcal{U}_n): n\in\mathbb{N}\}$ is a
countable base of $A$. Therefore $X$ is a $D_1$-space.
\end{proof}

\begin{theorem}\label{t4}
For a space $X$, $(CL(X), \tau_{loc fin})$ is first-countable if and
only if $X$ is a $D_1$-space and $NI(X)$ is hereditarily separable.
\end{theorem}

\begin{proof}
Suppose $(CL(X), \tau_{loc fin})$ is first-countable, then it follows from the
proof of (2) $\Rightarrow$ (4) of Theorem ~\ref{t3} that $X$ is a
$D_1$-space, hence $NI(X)$ is countably compact \cite[Theorem
1]{DL1995}. Let $A$ be a closed subset $NI(X)$, and let
$\{\mathcal{U}^-_n\cap \mathcal{U}^+_n: n\in\mathbb{N}\}$ be a
countable base of $A$, where each $\mathcal{U}_n$ is locally finite.
Since every element of $\mathcal{U}_n$ meets $A$ and $A$ is
countably compact, each $\mathcal{U}_n$ is finite. Hence let
$\mathcal{U}_n=\{U_{n, i}: i\leq m_n\}$ for each $n\in\mathbb{N}$.
Pick $x_{n, i}\in A\cap U_{n, i}$ for any $n\in\mathbb{N}, i\leq
m_{n}$, and then put $E=\{x_{n, i}: n\in\mathbb{N}, i\leq m_n\}$. It
is straightforward to prove that $E$ is a countable dense subset of
$A$. Therefore, $NI(X)$ is hereditarily separable.

Let $X$ be a $D_1$-space and $NI(X)$ is hereditarily separable. Fix
any $A\in CL(X)$. Let $\{U_n: n\in\mathbb{N}\}$ be a countable
strictly decreasing base of $C=A\cap NI(X)$ with $U_1=X$, and let
$E_n=(U_n\setminus U_{n+1})\cap (A\cap I(X))$ for $n\in\mathbb{N}$;
then each $E_n$ is discrete in $X$ and $A\cap I(X)=\bigcup_{n}E_n$.
Clearly, $I(X)\cap A=\bigcup_{n} (A\cap E_n)$ and each $E_{n}$ is
countable, hence we enumerate each $E_n$ as $\{x(n, \alpha(n)):
\alpha(n)\in \Gamma_n\}$, where each $|\Gamma_n|\leq\omega$. Let
$D=\{d_i: i\in \mathbb{N}\}$ be a countable dense subset of $C$, and
let $\mathcal{V}_i$ be a countable base of $d_i$ for each
$i\in\mathbb{N}$. Then $\mathcal{V}=\bigcup_n\mathcal{V}_n$ is a
countable family. For any $n, m\in\mathbb{N}$ and $\{W_1, ...,
W_r\}\in\mathcal{V}^{<\omega}$, put $$\mathcal{V}_{n, m, \{W_1, ...,
W_r\}}=\{\{x(i, \alpha(i))\}: \alpha(i)\in \Gamma_i, i\leq n,
\}\cup\{U_{m}, A\cap I(X)\}\cup\{W_1, ..., W_r\}.$$ Then put
$$\mathcal{B}=\{\mathcal{V}_{n, m, \{W_1, ..., W_r\}}^{+}\cap
\mathcal{V}_{n, m, \{W_1, ..., W_r\}}^{-}: n, m\in\mathbb{N}\
\mbox{and}\ \{W_1, ..., W_r\}\in\mathcal{V}^{<\omega}\}.$$  Clearly,
each family $\mathcal{V}_{n, m, \{W_1, ..., W_r\}}$ is locally
finite which consists of open subsets of $X$.

\smallskip
{\bf Claim:} $\mathcal{B}$ is a local base of $A$.

\smallskip
Let $\mathcal{P}^-\cap\mathcal{P}^+$ be an arbitrary open
neighborhood of $A$ in $(CL(X), \tau_{loc fin})$, where
$\mathcal{P}$ is a locally finite family of open subsets. Since $C$
is countably compact, there is a finite subfamily
$\mathcal{P}'=\{P_i: i\leq k\}$ of $\mathcal{P}$ such that each
element of $\mathcal{P}'$ meets $C$, $C\subset\bigcup_{i\leq k}P_i$
and $P\cap C=\emptyset$ for any $P\in
\mathcal{P}\setminus\mathcal{P}'$, there exists $m_{0}\in\mathbb{N}$
such that $U_{m_{0}}\subset \bigcup_{i\leq k}P_i$.  For each $i\leq
k$, $P_i$ contains some $d_{n_{i}}$, then pick
$W_i\in\mathcal{V}_{n_{i}}$ such that $d_{n_{i}}\in W_i\subset P_i$,
then $\{W_1, ..., W_k\}\in \mathcal{V}^{<\omega}$. Since
$A\in\mathcal{P}^-\cap\mathcal{P}^+$ and $P\cap C=\emptyset$ for any
$P\in \mathcal{P}\setminus\mathcal{P}'$, it follows that each $P\in
\mathcal{P}\setminus\mathcal{P}'$ meets some $E_{n}$. Indeed, we
claim that there exists $n_{0}\in \mathbb{N}$ such that each $P\in
\mathcal{P}\setminus \mathcal{P}'$ meets some $E_i$ for some $i\leq
n_{0}$. Suppose not, there exists a countable family of $\{P_l: l\in
\mathbb{N}\}\subset \mathcal{P}\setminus \mathcal{P}'$ such that
$E_{n_l}\cap P_l\neq \emptyset$ for each $l\in \mathbb{N}$. Pick any
$x_l\in E_{n_l}\cap P_l$ for each $l\in\mathbb{N}$; then $\{x_l:
l\in\mathbb{N}\}$ has a cluster point since $x_l\in U_{n_l}$. On the
other hand, $x_l\in P_l$ for each $l\in\mathbb{N}$ and $\{P_l:
l\in\mathbb{N}\}$ is locally finite, hence $\{x_l: l\in\mathbb{N}\}$
is discrete, this is a contradiction. Then $A\in\mathcal{V}_{n_{0},
m_{0}, \{W_1, ..., W_r\}}^{+}\cap \mathcal{V}_{n_{0}, m_{0}, \{W_1,
..., W_r\}}^{-}\subset \mathcal{P}^-\cap\mathcal{P}^+$ by
\cite[Lemma 1.1]{NS1988}.
\end{proof}

By Theorems~\ref{t3},~\ref{t4} and \cite[Theorem 7]{DL1995} and
Remark of Lemma ~\ref{l4}, the following corollary is easily
verified.

\begin{corollary}
The following statements are equivalent for a space $X$.
\begin{enumerate}
\item $(CL(X), \tau_{loc fin})$ is metrizable;

\smallskip
\item $(CL(X), \tau_{loc fin})$ has a point-countable base;

\smallskip
\item $(CL(X), \tau_{loc fin})$ is quasi-metrizable;

\smallskip
\item each closed countably compact subset of $(CL(X), \tau_{loc fin})$ has a countable base;

\smallskip
\item $X$ is metrizable and $NI(X)$ is compact.
\end{enumerate}
\end{corollary}

A family $\mathcal{B}$ of open subsets is called an {\it external $\pi$-base} of a subset $A$ of $X$, if for any open subset $U$ with $U\cap A\neq \emptyset$, there is a $B\in\mathcal{B}$ such that $B\cap A\neq \emptyset$ and $B\cap A\subset U$.
\begin{proposition}
For a space $X$, $(CL(X), \tau_{loc fin})$ has a point-$G_\delta$-property if and only if $X$ is perfect and each closed subset of $X$ has a $\sigma$-locally finite external $\pi$-base.

\end{proposition}

\begin{proof}
Necessity. Fix $A\in CL(X)$; then $\{A\}=\cap_{n\in\mathbb{N}}(\mathcal{U}^-_n\cap \mathcal{U}^+_n)$, where each $\mathcal{U}_n$ is a locally finite family consisting open subsets of $X$. Let $U_n=\bigcup\mathcal{U}_n$, then it is straightforward to prove that $A=\cap_{n\in\mathbb{N}}U_n$, hence $X$ is perfect. Let $\mathcal{U}=\bigcup_{n\in\mathbb{N}}\mathcal{U}_n$. Then $\mathcal{U}$ is a $\sigma$-locally finite external $\pi$-base of $A$. Indeed, let $V$ be an open subset of $X$ with $V\cap A\neq \emptyset$, $C=A\setminus V$, $C\neq A$. If for $U\in \mathcal{U}$, $U\cap A$ is not contained in $V$, then $U\cap C\neq \emptyset$ since $A\cap U\neq\emptyset$. Then $\{C\}\in \mathcal{U}^-_n\cap\mathcal{U}^+_n$ for all $n\in\mathbb{N}$, this is a contradiction. Hence $\mathcal{U}$ is an external $\pi$-base.

Sufficiency. Fix $A\in CL(X)$, and let $\mathcal{U}=\bigcup_{n\in\mathbb{N}}\mathcal{U}_n$ be a $\sigma$-locally finite external $\pi$-base, where $\mathcal{U}_n=\{U_\alpha: \alpha\in \Gamma_n\}$. Since $X$ is perfect, there is a countable open family $\{V_n: n\in\mathbb{N}\}$ such that $A=\bigcap_{n\in\mathbb{N}}V_n$. Let $\mathcal{U}_n\cap V_n=\{U_\alpha\cap V_n: U_\alpha\in \mathcal{U}_n\}$ and  $\mathcal{W}_n=(\mathcal{U}_n\cap V_n)^-\cap V^+$. Then $\{A\}=\bigcap_{n\in\mathbb{N}}\mathcal{W}_n$. Indeed, it is obvious that $\{A\}\in \mathcal{W}_n$ for each $n\in\mathbb{N}$. For any $\{B\}\neq \{A\}$, we prove that $\{B\}\notin \bigcap_{n\in\mathbb{N}}\mathcal{W}_n$. Indeed, if $\in B\setminus A\neq \emptyset$, pick any $y\in B\setminus A$; then there exists $n_0$ such that $y\notin V_{n_0}$, then $\{B\}\notin \mathcal{W}_{n_0}$. If $B\subset A$, then $(X\setminus B)\cap A\neq\emptyset$. Hence there is $U\in \mathcal{U}_{n'}$ such that $U\cap A\subset X\setminus B$ for some $n'$, $\{B\}\notin \mathcal{W}_{n'}$. Hence $\{A\}=\bigcap_{n\in\mathbb{N}}\mathcal{W}_n$.

A space $X$ is called {\it cosmic} if $X$ has a countable network.

\end{proof}
\begin{proposition}
$(CL(X), \tau_{loc fin})$ is cosmic if and only if $X$ is compact
metrizable.
\end{proposition}

\begin{proof}
If $X$ is compact, metrizable, then $(CL(X), \tau_{loc fin})$ is
compact, metrizanble, hence it is cosmic.

If $(CL(X), \tau_{loc fin})$ is cosmic, then $X$ is cosmic. We claim
that $X$ is countably compact, suppose not, $X$ contains a closed
countable discrete subset $D\subset X$. For any $\{A\}\in (CL(D),
\tau_{loc fin})$, $\{A\}$ is open in $(CL(D), \tau_{loc fin})$,
however $|(CL(X), \tau_{loc fin})|=2^\omega$, then it is not a
cosmic space, this is a contradiction. Hence $X$ is countably
compact, therefore $X$ is compact metrizable since a cosmic space
has a $G_\delta$-diagonal and a countably compact space with a
$G_\delta$-diagonal is compact metrizable \cite[Theorem
2.14]{G1984}.
\end{proof}

Finally, we prove the third main theorem of our paper.

\begin{theorem}
The following statements are equivalent for a space $X$.
\begin{enumerate}
\item $(CL(X), \tau_F)$ is metrizable;

\smallskip
\item $(CL(X), \tau_F)$ is quasi-metrizable;

\smallskip
\item $(\mathcal{K}(X), \tau_F)$ is quasi-metrizable;

\smallskip
\item $(\mathcal{F}(X), \tau_F)$ is quasi-metrizable;

\smallskip
\item $(\mathcal{F}_2(X), \tau_F)$ is quasi-metrizable;

\smallskip
\item $X$ is hemicompact and metrizable.
\end{enumerate}

\end{theorem}

\begin{proof}
The implications (1) $\Rightarrow$ (2) $\Rightarrow$ (3)
$\Rightarrow$ (4) $\Rightarrow$ (5) are trivial. By \cite[Theorem
5.1.5]{B1993}, we have (6) $\Rightarrow$ (1). We only need to prove
(5) $\Rightarrow$ (6).

Since $\mathcal{F}_2(X)$ is quasi-metrizable, it follows that $X$ is first-countable.
Fix any $z\in X$. Firstly, we prove that $X$ is locally compact at
$z$. Pick any $x\in X\setminus \{z\}$ and let $\{{U}^-_n\cap (K^c_n)^+ \cap
\mathcal{F}_2(X): n\in \mathbb{N}\}$ be a countable base of $\{x\}$,
where each $U_{n}$ is an open neighborhood of $x$ in $X$ and each
$K_{n}$ is compact in $X$. For any compact subset $K\subset
X\setminus\{x\}$, the set $(K^c)^+\cap\mathcal{F}_2(X)$ is an open
neighborhood of $\{x\}$ in $\mathcal{F}_2(X)$, hence there is
$n\in\mathbb{N}$ such that $U_{n}^-\cap (K^c_n)^+\cap
\mathcal{F}_2(X)\subset (K^c)^+\cap\mathcal{F}_2(X)$. If $K\setminus
K_n\neq\emptyset$, then pick any $y\in K\setminus K_n$. Clearly, $\{x,
y\}\in U_{n}^-\cap (K^c_n)^+\cap \mathcal{F}_2(X)$; however, $\{x,
y\}\notin (K^c)^+$, which is a contradiction. Therefore, $K\subset
K_n$. Then $X\setminus \{x\}$ is hemicompact and $X$ is
$\sigma$-compact. Hence $X\setminus\{x\}$ is locally compact by
\cite[3.4.E]{E1989}, and $X$ is locally compact at $z$. Therefore
$X$ is locally compact. Let $U$ be an open neighborhood of $x$ with
$\overline{U}$ compact; then we prove any compact subset $L$ of $X$, the set $L$
is covered by a finite subfamily of $\{\overline{U}, K_n: n\in
\mathbb{N}\}$. Indeed, If $L\subset X\setminus \{x\}$, then $L$ is
contained in the union of a finite subfamily of $\{K_n:
n\in\mathbb{N}\}$; if $x\in L$, then $L\setminus U\subset X\setminus
\{x\}$ is compact, which is covered by a finite subfamily
$\mathcal{K}'$ of $\{K_n: n\in \mathbb{N}\}$, hence $L\subset
U\cup(\bigcup\mathcal{K}')$. $X$ is hemicompact.

Note that a quasi-metrizable space is a $\gamma$-space, then it
follows from \cite[Corollary 10.8(ii)]{G1984} that $X$ is locally
metrizable, hence it is metrizable since $X$ is Lindel\"{o}f.
\end{proof}

\end{document}